\documentclass[a4paper,12pt]{amsart}
\usepackage[paper=a4paper,left=25mm,right=25mm,top=25mm,bottom=25mm]{geometry}

\usepackage{amscd}
\usepackage{bbm}
\usepackage{mathtools}
\usepackage{enumitem}

\allowdisplaybreaks[1]

\newtheorem{Theorem}{Theorem}[section]
\newtheorem{Lemma}[Theorem]{Lemma}

\newtheorem{Remark}[Theorem]{Remark}

\theoremstyle{definition}
\newtheorem{Definition}[Theorem]{Definition}

\usepackage{amssymb}

\newcommand\C{{\mathbb C}}

\newcommand\N{{\mathbb N}}

\newcommand\la{\lambda}

\newcommand\Ga{\Gamma}
\newcommand\ga{\gamma}

\newcommand\lu{\left\langle}
\newcommand\ru{\right\rangle}

\begin{document}
\thispagestyle{empty}
\title{Random zeros entire functions}

\author{\textbf{Yuri Kondratiev}\\
 Dragomanov University, Kyiv, Ukraine}
 
 \begin{abstract}

We propose the construction of entire functions
with a given random collection of zeros. There are
considered two particular cases. In the first one we are
dealing with simple zeros.  And the second corresponds 
to random zeros with random multiplicity.



 \end{abstract}

\maketitle

\section{Introduction}

A classical problem in complex analysis consists in  a construction 
of an entire function with a prescribed set of zeros. For a finite set
of zeros the answer is given by a proper polynomial. More generally, 
let $a_1, a_2, \dots$ be a given sequence of complex numbers, non of which is zero,
with the point at infinity as the only limit point. We would like construct an entire function
whose set of zeros is precisely this sequence, see, e.g., \cite{Levin} for the detailed discussion.

\begin{Definition}

Put $E_0(z) = 1-z$, and for $p\in \N$
$$
E_p(z)= (1-z) \exp\{z+\frac{z^2}{2} + \dots + \frac{z^p}{p}\}.
$$
These functions are called elementary factors.

\end{Definition}

We form an infinite product
\begin{equation}
\label{prod}
\Pi(z)= \prod_{n=1}^\infty E_{p_n} (\frac{z}{a_n}).
\end{equation}
There we have chosen a sequence of natural numbers $p_n$ such that 
$$
\sum_{n=1}^\infty \frac{|z|^{p_n +1}}{|a_n|^{p_n+1}} < \infty.
$$
Such sequence always exists \cite{Levin}. 

Of course, the sequence $a_n, n\in \N$ is the set of zeros for
$\Pi(z)$.  Representation similar to (\ref{prod}) is possible for any entire
function that is stated by the Weierstrass theorem. If all $p_n=p\in \N, n\in\N$
then we are in the framework of the Hadamar theorem \cite{Levin}.

The aim of this short note is to show that the results above may be extended to
the case of random zeros. For the more transparent  presentation we consider 
only the case of Poisson random fields on $\C$. But the obtained results may be
stated for several other random point fields as, e.g., Gibbs point processes. 
Actually, as is may be seeing from the proofs below, only information we need
is the density of a random point process we use.

We will consider two particular cases. In the first one we are dealing with simple random zeros.
In the second are considered random zeros with random multiplicity.

\section{Poisson zeros }

We will use Poisson random fields over the complex plane $\C$. To this end we introduce the space $\Ga(\C)$ of all locally 
finite subsets $\ga\subset \C$ which are called configurations. For a detailed description of topological and metrical
properties of $\Ga(\C)$ see, e.g., \cite{AKR}, \cite{FKLO}, \cite{KKO}. Let $\sigma$ be a Radon measure on $\C$. In particular
we will use the case of the Lebesgue measure $m(dz)= dxdy,\;\; z=x+iy$. 

For $f\in C_0(\C)$ (the set of continuous fuctions with compact supports) 
introduce a lifted function on $\Ga(\C)$ by
$$
\lu f,\ga\ru =\sum_{x\in \ga} f(x), \;\; \ga \in \Ga(\C).
$$
The Poisson measure $\pi_\sigma$ on $\Ga(\C)$ is defined via the Laplace transform
$$
\int_{\Ga(\C)} e^{\lu f,\ga\ru} d\pi_\sigma (\ga)= \exp(\int_\C\{ e^{f(x)} -1)\}d\sigma(x)).
$$

As a corollary of this definition we may show for any
$f\in L^1(\C,\sigma)$
\begin{equation}
\label{linear}
\int \lu f,\ga\ru d\pi_\sigma = \int_\C f(z) d\sigma(z).
\end{equation}

For any $R>0$ denote $B_R=\{z\in\C \;|\; |z|\leq R \}$ and $B_R^c= \{z\in\C\;|\; |z|>R\}.$
Then $\C= B_R\cup B_R^c$ and 
$$
\Ga(\C) = \Ga(B_R) \times \Ga(B_R^c), \; \ga= \ga_R \cup \ga_R^c.
$$ 
Note that $\Ga(B_R)$ contains only finite configurations as $B_R$ is compact.
For the decomposition $\C= B_R\cup B_R^c$ we have the decomposition   of the measure
$$
\sigma= \sigma_R + \sigma_R^c
$$
and corresponding Poisson measures product 
$$
\pi_\sigma = \pi_{\sigma_R} \otimes \pi_{\sigma_R^c}.
$$

\begin{Lemma}[\cite{Rudin}]

\label{RUD}

For $|z|<1$ 

\begin{equation}
\label{bound}
|1-E_p(z)| \leq |z|^{p+1}.
\end{equation}
\end{Lemma}

\begin{Lemma}[\cite{Levin}]
\label{bound+}
For all $z\in \C$ 
\begin{equation}
\ln |E_p(z)|< A_p \frac{|z|^{p+1}}{1+|z|}.
\end{equation}

\end{Lemma}

Our aim is the following: for fixed $p$ and a configuration $\ga\in \Ga(\C)$ we consider  a heuristic expression
$$
\Pi (z,\ga)= \prod_{x\in \ga, x\neq 0} E_p(\frac{z}{x}).
$$
Assume we take $z\in B_R$. We can decompose 
$$
\Pi (z,\ga)= \prod_{x\in \ga, x\neq 0} E_p(\frac{z}{x})= \prod_{x\in \ga_R, x\neq 0} E_p(\frac{z}{x}) \times
 \prod_{x\in \ga_R^c} E_p(\frac{z}{x})
 $$
 and the existence of this function depends only on the second (infinite product) factor.
 
 \begin{Theorem}
 For any $p\geq 2$ the function $\Pi_p(z,\ga)$ is defined for $\pi_m$-a.a. $\ga\in \Ga(\C)$ as
 an entire function of $z\in\C$.
 \end{Theorem}
 
 \begin{proof}
 
 We rewrite 
 $$
 \Pi_p^{R}(z,\ga)= \prod_{x\in\ga_R^c} E_p(\frac{z}{x}) = \prod_{x\in\ga_R^c} (1 - (1-E_p(\frac{z}{x}))
 $$ and the convergence of the infinite product is equivalent to the convergence of the sum
 $$
 \sum_{x\in\ga_R^c} (1-E_p(\frac{z}{x})).
 $$
 
 Now we will use relation (\ref{linear}) with $f(z)= (1-E_p(\frac{z}{x}))$. We have
 $$
 \int_{\Ga(B_R^c)}  \sum_{x\in\ga_R^c} (1-E_p(\frac{z}{x})) d\pi_{R}^c(\ga_R^c) =
 $$
 $$
 \int_{B_R^c}  (1-E_p(\frac{z}{x})) dm_{R^c}(z) \leq \int_{B_R^c}  \frac{|z|^{p+1}}{|x|^{p+1}} dm_{R^c} (x)=
 C|z|^{p+1} \int_R^\infty \frac{1}{\rho^p} d\rho
 $$
 where we used bound (\ref{bound}). For $p\geq 2$ last integral is finite that means
 the integrability of $ \sum_{x\in\ga_R^c} (1-E_p(\frac{z}{x}))$ and, therefore, this sum is
 finite  $\pi_{m_{R}^c}$ almost  surely. Then 
 $$
 \sum_{x\in\ga} (1-E_p(\frac{z}{x}))
 $$
 converges $\pi_m$ almost surely and uniformly w.r.t. $z\in B_R$.

 Using Theorem 15.6 from \cite{Rudin} we conclude that $E_p(z,\ga)$ is well defined for a.a. $\ga\in\Ga(\C)$ 
 and is holomorphic on $B_R$.  But it is true for all $R>0$ that means $E_p(z,\ga)$ is entire.
 
 \end{proof}

 \begin{Remark}
 An alternative proof of this theorem is based on the analysis of
 $$
 \log \Pi_p (z,\ga)= \sum_{x\in \ga, x\neq 0} \log E_p(\frac{z}{x}).
 $$
 Using bound (\ref{bound+}) we again arrive in the necessity to estimate the convergence of
 $$
 \sum_{x\in \ga, x\neq 0} \frac{1}{|x|^{p+1} }
 $$
 and then we proceed as above.
 \end{Remark}

 \begin{Remark}
 We have a straightforward generalization of this theorem to the case when
 instead of the Lebesgue measure $m$ we use another intensity measure
 $\sigma$ for the Poisson random field.
 \end{Remark}

\begin{Remark}
It clear that the function $\Pi_p(z,\ga)$ has $\ga$ at the set of simple zeros.
\end{Remark}

\section{Marked Poisson zeros}

Now we will consider the case of random zeros with random multiplicity.
Technically, this situation may be described by means of marked random point fields.

Let  us  introduce the marked configuration space $\Ga(\C,\N)$ with marks from 
$\N$.  For $x\in \C$ we consider a pair $(x,n_x)$ with $n_x\in \N$. 
The set of all such pairs $\hat \ga=\{(x,n_x)|x\in \ga\in \Ga(\C)\} $ we denote
$\Ga(\C,\N)$. Obviously, 
$$
\Ga(\C,\N) \subset \Ga(\C \times \N).
$$
For a given probability measure $\la$ on $\N$ consider the Poisson measure $\pi_{m\otimes\la}$ on $\Ga(\C\times \N)$
and denote $\pi_m^\la$ its restriction on $\Ga(\C,\N)$. It is possible because the last set has the full measure, see
\cite{KKS} for detailed study of marked random fields.

For a marked configuration $\hat \ga$ consider 
$$
\Pi_p(z,\hat \ga) = \prod_{x\in\ga, x\neq 0} \left( E_p(x, \frac{z}{x} \right )^{n_x}
$$
that takes into account the multiplicity $n_x$ for  each zero $x\in \ga$.

\begin{Theorem}
Assume that 
$$
\int_\N n d\la(n)= \sum_{n\in\N} n \la(\{n\}) <\infty.
$$
For each $p\geq 2$ and $\pi_m^\la$-a.a. $\hat\ga\in \Ga(\C,\N)$ 
we have that
$\Pi_p (z, \hat \ga)$  exists as an entire function
of $z\in\C$.

\end{Theorem}

\begin{proof}
Consider $z\in B_R, R>0$. For
$$
\Pi_p^R (z,\hat \ga)= \prod_{x\in\ga, |x|>R} E^{n_x}_p(z/x)
$$
holds
$$
\log \Pi_p^R (z,\hat \ga) =\sum_{x\in\ga, |x|>R} n_x \log E_p(z/x)
$$
and then
$$
\int \log \Pi_p^R (z,\hat \ga) d\pi_m^\la (\hat \ga) = \int \left(  \sum_{x\in\ga, |x|>R}  n_x  \log E_p(z/x)\right) d\pi_m^\la(\hat\ga)=
$$
$$
\int_\N \int_{|x|>R} nE_p(z/x) d\la(n)dm(x)\leq \int_\N  n d\la(n) \int_{|x|>R} \frac{|z|^{p+1}}{|x|^{p+1}} dm(x)=
$$
$$
C\int_{R}^\infty \frac{1}{\rho^p} d\rho < \infty.
$$
We have used the bound (\ref{bound+}).  Therefore, 
$$
\log \Pi_p (z,\hat \ga) =\sum_{x\in\ga} n_x \log E_p(z/x)
$$
is finite $\pi_m^\la$ almost surely and converges uniformly in $z\in B_R$. As above it gives the statement 
of the theorem.

\end{proof}

\vspace{5mm}
\section{Acknowledgment}

The financial support by the Ministry 
for Science and Education of Ukraine
through Project 0122U000048
is gratefully acknowledged.

\end{document}